\documentclass[11pt]{article}

\usepackage{enumerate}
\usepackage{mathrsfs}
\usepackage{amsmath, color}
\usepackage{latexsym}
\usepackage{amssymb}
\usepackage{amsthm}
\usepackage{amsfonts}
\usepackage{dsfont}
\usepackage{hyperref}

\renewenvironment{proof}{{\noindent \rm Proof.}}{\qed \\ \\}

\theoremstyle{plain}
\newtheorem{thm}{Theorem}[section]
\newtheorem{prop}[thm]{Proposition}

\newtheorem{lem}[thm]{Lemma}

\newtheoremstyle{osaka}
  {\topsep}
  {\topsep}
  {\itshape}
  {0pt}
  {\rmfamily}
  {.}
  { }
  {\thmname{#1}\thmnumber{ #2}\thmnote{ #3}}

\theoremstyle{osaka}

\newtheorem{assumption}[thm]{Assumption}
\newtheorem{rmk}[thm]{Remark}

\numberwithin{equation}{section}

\newcommand{\IP}{{\mathbb{P}}}
\newcommand{\IR}{{\mathbb{R}}}
\newcommand{\FF}{{\mathcal{F}}}
\newcommand{\EE}{{\mathcal{E}}}

\usepackage{lipsum}

\newcommand\blfootnote[1]{%
  \begingroup
  \renewcommand\thefootnote{}\footnote{#1}%
  \addtocounter{footnote}{-1}%
  \endgroup
}

\title{\bf On-diagonal Heat Kernel Lower Bound for Strongly Local Symmetric Dirichlet Forms}
\date{\today}
\author{{\bf Shuwen Lou}}
\begin{document}

\maketitle
\begin{abstract}
\noindent This paper studies strongly local symmetric Dirichlet forms on general measure spaces. The underlying space is equipped with the intrinsic metric induced by the Dirichlet form, with respect to which the metric measure space does not necessarily satisfy volume-doubling property. Assuming Nash-type inequality, it is proved in this paper that outside a properly exceptional set, given a pointwise on-diagonal heat kernel upper bound in terms of the volume function, the comparable heat kernel lower bound also holds.  The only assumption made on the volume growth rate   is that it can be bounded by a continuous function satisfying  doubling property, in other words, is not exponential.
\end{abstract}
\noindent

\blfootnote{ AMS 2010 Mathematics Subject Classification: Primary 31C25, 35K08; Secondary 60J35, 60J45.}



\section{Introduction}

 It has been well-known that Dirichlet forms provide an elegant way to characterize Markov processes. Any regular symmetric Dirichlet form admits a symmetric Hunt process (see, for instance, \cite[Theorem 1.5.1, Theorem 3.1.12]{CF}) associated with it. Furthermore, there is a one-to-one correspondence between the family of strongly local regular symmetric Dirichlet forms and the family of diffusion processes with no killing inside. In recent years, Dirichlet form theory has been serving as a powerful tool to construct processes on irregular spaces. For instance, varieties of strong Markov processes with darning have been constructed by Chen and Fukushima in \cite{Chen, CF}, including one-dimensional absorbing Brownian motion, circular Brownian motion, knotted Brownian motion, multidimensional Brownian motion with darning, diffusions on half-lines merging at one point , etc. In my recent joint work \cite{CL1} with Chen, Brownian motion on spaces with varying dimension are characterized in terms of Dirichlet form with darning and therefore have been studied with an emphasis on their two-sided heat kernel behaviors. One of the major difficulties in studying processes constructed on irregular spaces is to describe their behaviors near the singularities.

It therefore becomes natural to ask whether there is any {\it general} method or criterion treating heat kernel bounds that can be applied to 
Dirichlet forms constructed on state spaces that possibly contain singularities thus do not allow any of the typical methods to work, such as parabolic Harnack inequality, Poincar\'{e} inequality, or volume-doubling property? For example, in \cite{CL1}, none of these properties holds for Brownian motion on spaces with varying dimension due to the inhomogeneity at the darning point(s).

The amount of existing literature answering the question above is very limited. 
The established results on heat kernel estimates are mostly under the frameworks of either Laplace-Beltrami operators on Riemannian manifolds (for example, \cite{LSC1, LSC2}), or Dirichlet forms on metric measure spaces
 (for example, \cite{GH, GHL, GT}).  The majority of these existing results require the underlying spaces to satisfy volume-doubling or other regularity conditions. In this paper,  the underlying space is not necessarily equipped with an original metric. Instead, we equip it with the intrinsic metric induced by the Dirichlet form. Without assuming volume-doubling property of the unerlying measure with respect to the intrinsic metric, we give sharp on-diagonal heat kernel lower bound for general strongly local regular symmetric Dirichlet forms.

A similar problem has been answered in \cite{CG} by Coulhon and Grigor'yan, which gives criteria for pointwise on-diagonal two-sided heat kernel bounds associated with Laplace-Beltrami operators on weighted Riemannian manifolds without assuming volume-doubling property. The two-sided bound only depends on the local volume form of the space near the particular point. The key ingredients in their paper are the integral estimations of the heat kernels and their time derivatives established in \cite[Theorem 1.1]{gowri} and \cite[Theorem 1]{G1}. Some analogous properties were earlier proved for the  fundamental solutions to parabolic equations  by  Aronson in \cite{Aronson}. In this paper, we also extend these integral estimates further to strongly local Dirichlet spaces.

Throughout this paper,  $(\EE, \FF)$ is a  Dirichlet form on a real Hilbert space $L^2(E,\mu)$. The underlying space $E$ is a locally compact separable Hausdorff space equipped with a reference measure $\mu$ which is a positive Radon measure with full support on $E$. With the norm $\|u\|_\FF:=(\EE(u, u)+\|u\|_{L^2})^{1/2}$, $\FF$ is also a real Hilbert space.  The Dirichlet form $\EE$ is assumed to be regular, symmetric, and strongly local.  A Dirichlet form $(\EE, \FF)$ is said to be regular if $C_c(E)\cap \FF$ is dense in $(\EE_1, \FF)$ and $(C_c(E),\|\cdot \|_\infty)$. It is symmetric if $\EE(u, v)=\EE(v, u)$ for any $u, v\in \FF$.  It is strongly local means $\EE(u,v)=0$ whenever $u$ is equal to a constant on a neighborhood of the support of $v$. In other words, $\EE$ has no killing or jumping part.

As usual we denote the infinitesimal operator associated with $\EE$ by $\mathcal{L}$.  It follows that the family of $\{P_t=e^{\mathcal{L}t}, \, t\ge 0\}$ is a strongly continuous semigroup on $L^2(E, \mu)$, and that there is a unique symmetric diffusion process $X$ associated with $(\EE, \FF)$ whose transition semigroup is $\{P_t\}_{t\ge 0}$. Furthermore, $X$ can start from every point of $E$ outside a properly exceptional set \footnote{A set $\mathcal{N}\subset E$ is called properly exceptional if it is Borel, $\mu(\mathcal{N})=0$ and $\IP_x(X_t\in \mathcal{N} \text{ for some }t\ge0)=0$ for all $x\in E\setminus \mathcal{N}$ (see \cite[p.134 and Theorem 4.1.1 on p.137]{FOT}).} denoted by $\mathcal{N}$.  A family of functions $\{p(t,x,y)\}_{t\ge 0, x,y\in E\setminus \mathcal{N}}$ is called the heat kernel of $(\EE, \FF)$ if for all $t>0$ and $\mu-$a.e. $y\in E$,
\begin{equation}\label{heat-kernel-definition}
P_tf(y)=\int_E p(t,x, y)f(x)\mu(dx), \quad \text{for every }f\in L^2(E).
\end{equation}

The main result of this paper is the following theorem.
\begin{thm}\label{main-result}             
Let $(\EE, \FF)$ be a strongly local regular symmetric Dirichlet form satisfying Assumption \ref{strong-regularity} and Nash-type inequality \eqref{Nash-inequality-I}. Fix $z\in E\setminus \mathcal{N}$ where $\mathcal{N}$ is a properly exceptional set.  With respect to the intrinsic metric induced  by $(\EE, \FF)$, assume that for all $r>0$, $\mu(B(z,r))\le v(r)$, where  $v(r)$ is a continuous monotonically increasing function satisfying doubling property in the following sense: There exists some $A>0$ such that 
\begin{equation*}\label{doubling}
v(2r)\le Av(r), \quad \text{for all }r>0.
\end{equation*}
Suppose also that for some $C_1>0$, $T\in (0, \infty]$,
\begin{equation*}\label{UPE}
p(t,z,z)\le \frac{C_1}{v(\sqrt{t})}, \quad t\in (0, T).
\end{equation*}
Then there exists $C_2>0$ such that 
\begin{equation*}\label{LE}
p(t,z,z)\ge \frac{C_2}{v(\sqrt{t})}, \quad t\in (0, T).
\end{equation*}
\end{thm}

Note that the definition of intrinsic metric is given in \eqref{def-instrinsic-metric}, and Assumption \ref{strong-regularity} is made based on that. The  intrinsic metric is the metric under which two-sided Gaussian-type heat kernel bounds can be characterized by parabolic Harnack inequality or the conjunction of volume-doubling property and Poincar\'{e} inequality. See \cite{St3}. Finally we briefly explain the necessity of imposing Assumption \ref{strong-regularity} and Nash-type inequality. Indeed, Nash-type inequality is a natural assumption to 
ensure that the heat kernel associated with the Dirichlet form exists.  Assumption \ref{strong-regularity}, on the other hand,  guarantees that intrinsic distance functions are non-degenerate and in the local Dirichlet form domain, and that cutoff distance functions (with respect to the intrinsic metric) are in $\FF$. More details will be given in Section \ref{S:2}.  For more delicate discussion on Assumption \ref{strong-regularity} and its variations, one may refer to \cite{Sturm1} and \cite{St3}.

The rest of the paper is organized as follows. In Section \ref{S:2}, we briefly introduce the definitions and some basic properties of the energy measures associated with strongly local regular symmetric Dirichlet forms. Then we give the definition of intrinsic metric induced by  Dirichlet forms. Using Nash-type inequality, we claim the existence of the heat kernel $\{p(t,x,y)\}_{t\ge 0, x,y\in E\setminus \mathcal{N}}$ and establish a rough off-diagonal heat kernel  upper bound which follows from Davies method. This (rough) upper bound does not need to be in the volume form $v(\sqrt{t})^{-1}$ along the diagonal $\{x,y\in E\setminus \mathcal{N}:\, x=y\}$.  Theorem \ref{main-result} will be proved in Section \ref{S:3}.

\section{Preliminary}\label{S:2}

It is known that any strongly local symmetric Dirichlet form $(\EE, \FF)$ can  be written in terms of the energy measure $\Gamma$ as follows:
$$
\EE(u, v)=\int_E d\Gamma(u, v), \quad u, v\in \FF.
$$
where $\Gamma$ is a positive semidefinite symmetric bilinear form on $\FF$ with values being signed Radon measures on $E$, which is also called the energy measure. To be more precise, we first define for every $u\in \FF\cap L^\infty(E)$ and every $\phi\in \FF\cap \mathcal{C}_0(E)$ 
$$
\int_E \phi \,d\Gamma(u, u)=\EE(u, \phi u)-\frac{1}{2}\EE(u^2, \phi). 
$$
The quadratic form $u\mapsto \Gamma(u, u)$ can be extended to $\FF$ using the approximation sequence $u_n:=n\wedge u\vee (-n)$. Recall that local Dirichlet space is defined as
\begin{align}
\FF_{\text{loc}}:=\{&u: \text{ for every relatively compact open set $D$, there exists some $v\in \FF$ such that } \nonumber
\\
&v|_{D}=u,\, \mu-a.e. \}.\label{def-local-Dirichlet-space}
\end{align}
It follows that every $u\in \FF_{\text{loc}}$ admits a $\mu$-version which is quasi-continuous\footnote{A function $f$ is called``$\EE$-quasi-continuous" if for any $\epsilon>0$, there is an open set $O$ with capacity less than $\epsilon$ such that $f|_{E\setminus O}$ is continuous (See \cite[\S 2.3 on p.77, Theorem 3.17 on p.96, and Theorem 3.3.3 on p.107]{CF}).}. Furthermore, the domain of the  map $u\mapsto \Gamma(u, u)$ can be extended to $\FF_{\text{loc}}$ (see \cite[Theorem 4.3.10(ii), p.248-249]{CF}). By polarization, for $u, v\in \FF_{\text{loc}}$,
\begin{equation*}
\Gamma(u, v):=\frac{1}{4}\left(\Gamma(u+v, u+v)-\Gamma(u-v, u-v)\right)
\end{equation*}
is defined as a signed Radon measure. 
The following Cauchy-Schwarz inequality is satisfied by energy measures,  which can be found in \cite[Appendix]{Sturm1}.
\begin{prop}\label{Cauchy-Schwarz}
Let \emph{$u, v\in \FF_{\text{loc}}$}. For $f, g\in L^\infty(E)$ that are quasi-continuous, it holds
\begin{align}
\int_E fgd\Gamma(u, v)&\le \left(\int_E f^2 d\Gamma (u, u)\right)^{\frac{1}{2}}\left(\int_E g^2 d\Gamma (v, v)\right)^{\frac{1}{2}} \le \frac{1}{2}\left(\int_E f^2 d\Gamma(u, u)+\int_E g^2d\Gamma(v, v)\right).\label{Cauchy-Schwarz-energy-measure}
\end{align}
\end{prop}

Energy measures satisfy the following properties called Leibniz rule and chain rule for strongly local Dirichlet spaces. See \cite[Appendix]{Sturm1} or \cite[Chapter 4]{CF}.
\begin{thm}\label{chain-product-rule} Let $(\EE, \FF)$ be a strongly local regular Dirichlet space. The following properties hold:
\begin{description}
\item{\emph{(i)}} For any \emph{$u,v\in \FF_{\text{loc}}\cap L^\infty(E)$} and \emph{$w\in \FF_{\text{loc}}$}, 
$$
d\Gamma(uv, w)=\tilde{u}d\Gamma(v, w)+\tilde{v}d\Gamma(u, w),
$$ where $\tilde{u}$, $\tilde{v}$ are the quasi-continuous versions of $u$ and $v$. 

\item{\emph{(ii)}}  For any $\Phi\in C_b^1(\IR)$ with bounded derivative $\Phi '$. Then \emph{ $u\in \FF_{\text{loc}}$} implies \emph{$\Phi(u)\in \FF_{\text{loc}}$} and
$$
d\Gamma( \Phi(u), v)=\Phi '(u)d\Gamma( u',v),
$$ 
for all \emph{$v\in \FF_{\text{loc}}\cap L^\infty_{\text{loc}}(E)$}.
\end{description}
\end{thm}

To introduce heat kernels associated with Dirichlet spaces, we first introduce the  intrinsic metric $d$ on $E$ induced by the energy measure $\Gamma$:
\begin{equation}\label{def-instrinsic-metric}
d(x, y):=\sup\{u(x)-u(y):\, u\in \FF_{\text{loc}}\cap \mathcal{C}(E), \, \Gamma(u, u)\le \mu \text{ on }E\},
\end{equation}
where $\Gamma(u, u)\le \mu$ should be interpreted as follows: The energy measure $\Gamma(u, u)$ is absolutely continuous with respect to the underlying measure $\mu$ with its Radon-Nikodym derivative $d\Gamma(u, u)/d\mu \le 1$ a.e.. In fact, the Radon-Nikodym derivative $d\Gamma(u, u)/d\mu$ should be interpreted as the square of the gradient of $u\in \FF_{\text{loc}}$.

Generally speaking, $d$ is a pseudo metric instead of a metric, which means it may be degenerate ($d(x,y)=0$ or $\infty$ for $x\neq y$). To ensure $d$ is a metric and all cutoff distance functions are in $\FF$, we make the following fundamental assumption throughout this paper:
\begin{assumption}\label{strong-regularity}
The topology induced by $d(\cdot, \cdot)$ in \eqref{def-instrinsic-metric}  coincides with the original one and all balls in the form of $B_r(z):=\{x\in E: d(x,z)<r\}$ are relatively compact.
\end{assumption}
This assumption in particular implies that $d$ is non-degenerate. The following fundamental lemma has been proved in \cite[Lemma 1]{Sturm1} for strongly local regular Dirichlet spaces satisfying Assumption \ref{strong-regularity}.
\begin{lem}\label{property-energy-meas}
Let $(\EE, \FF)$ be a strongly local regular symmetric Dirichlet form satisfying Assumption \ref{strong-regularity}. For every $z\in E$, the distance function $d_z(x): \, x\mapsto d(x,z)$ is in \emph{$\FF_{\text{loc}}\cap \mathcal{C}(E)$} and satisfies
\begin{equation}\label{L:2.4}
\Gamma (d_z, d_z)\le \mu.
\end{equation}
\end{lem}
\begin{rmk}\label{R:2.5}
Indeed, Lemma \ref{L:2.4} holds without the assumption that all balls are relatively compact. However, given that all open balls are relatively compact, Lemma \ref{L:2.4} implies that all cutoff functions: $x\mapsto (r-d(x,z))_+$ are in $\FF\cap L^\infty$ and satisfy \eqref{L:2.4}.
\end{rmk}

We also assume throughout this paper that the Dirichlet form $(\EE, \FF)$ satisfies the following Nash-type inequality: There exist some $\gamma>0$,  $\delta\ge 0$ and some $A>0$,
\begin{equation}\label{Nash-inequality-I}
\|f\|_2^{2+4/\gamma}\le A\left(\EE(f, f)+\delta \|f\|_2^2\right)\|f\|_1^{4/\gamma}, \quad \text{for all }f\in \FF.
\end{equation}

The  existence of the heat kernel along with its short time off-diagonal estimate follows immediately from Nash inequality, as the next proposition claims. 
\begin{prop}\label{P:upper-bound-off-diag-Nash}
Let $(\EE, \FF)$ be a strongly local regular symmetric Dirichlet form satisfying  Nash-type inequality \eqref{Nash-inequality-I} with some $\gamma>0, \delta\ge 0$ and $A>0$.  
  \begin{description}
\item{\emph{(i)}} There is a properly exceptional set $\mathcal{N}\subset E$ of $X$ such that there is a positive symmetric kernel $p(t,x,y)$ defined on $(0, \infty)\times (E\setminus \mathcal{N})\times (E\setminus \mathcal{N})$ satisfying \eqref{heat-kernel-definition} and 
\begin{equation*}
p(t+s, x, y)=\int_{E} p(t,x,z)p(s, z, y)\mu(dy), \quad t, s>0, \, x, y\in (E\setminus \mathcal{N}).
\end{equation*}
Additionally, for every $t>0$, $y\in E\setminus \mathcal{N}$, the map $x\mapsto p(t,x,y)$ is quasi-continuous on $E$. 
\item{\emph{(ii)}} There exist $C_1, C_2>0$ such that for every $x\in E\setminus \mathcal{N}$,
\begin{equation}\label{e:2.6}
p(t,x,y)\le \frac{C_1}{t^{\gamma/2}}e^{-C_2d(x,y)^2/t}, \quad t\in (0, 1], \, \mu-a.e. \, y,
\end{equation}
for the same $\gamma$  as in Nash inequality \eqref{Nash-inequality-I}.
\end{description}
\end{prop}
\begin{proof}
It follows from \cite[Theorem 2.1]{CKS} that for some $c_1>0$, 
\begin{equation*}
\|P_t f\|_\infty \le \frac{c_1 e^{\delta t}}{t^{\gamma /2}} \|f\|_1, \quad f\in L^1(E), \, t>0.
\end{equation*}
Therefore (i) follows immediately from \cite[Theorem 3.1]{BBCK}. The proof to (ii) follows from a standard argument using Davies's method: 
Fix $x_0, y_0\in E\setminus N$, $0<t_0\le 1$. Set a constant $\alpha:=d(y_0,x_0)/4t_0$ and
$\displaystyle{\psi (x):=\alpha \cdot d(x, x_0)}$.
Then we define $\psi_n(x)=\psi(x)\wedge n$. Note that for $\mu$-a.e. $x\in E\setminus \mathcal{N}$, by Theorem \ref{chain-product-rule} and Lemma \ref{property-energy-meas},
\begin{displaymath}
e^{-2\psi_n (x)} \frac{d}{d\mu}\Gamma\left( e^{\psi_n (x)}, e^{\psi_n(x)} \right)=\frac{d}{d\mu}\Gamma \left(\psi_n (x), \psi_n(x)\right) \leq \alpha^2.
\end{displaymath}
Similarly,   
$$ 
e^{2\psi_n (x)} \frac{d}{d\mu}\Gamma\left( e^{-\psi_n (x)}, e^{-\psi_n(x)} \right) \leq \alpha^2.
$$
By \cite[Theorem 3.25]{CKS}, there exists some $c_2>0$ such that for every $x\in E\setminus \mathcal{N}$,
\begin{equation*}
p(t,x,y)\le \frac{c_2e^{2\delta t}}{t^{\gamma/2}}\exp\left(-|\psi(y)-\psi(x)|+2t|\alpha|^2\right), \quad t>0, \, \mu-a.e.\, y.
\end{equation*}
i.e., 
\begin{equation}\label{proveoffdiagUB}
p(t,x,y)\le \frac{c_2e^{2\delta }}{t^{\gamma/2}}\exp\left(-|\psi(y)-\psi(x)|+2t|\alpha|^2\right), \quad 0<t\le 1, \, \mu-a.e.\, y.
\end{equation}
  Taking $t=t_0, x=x_0$
and $y=y_0$ in \eqref{proveoffdiagUB} completes the proof.
\end{proof}

With the above heat kernel upper bound, the following well-known result can be justified using spectral theory. See \cite[Example 4.10]{GH}.

\begin{lem}\label{L:2.7}
Fix $y\in E\setminus \mathcal{N}$. For every $t>0$, the map $x\mapsto p(t,x,y)$ is in domain of the infinitesimal generator $\mathcal{L}$ and satisfies the heat equation
\begin{equation*}
\partial_t p(t,x,y)+\mathcal{L}p(t,x,y)=0,
\end{equation*}
where $\partial_t p(t,x,y)$ is the strong derivative of the map $t\mapsto p(t,x,y)$ in $L^2$. 
\end{lem}
This immediately yields that the map $x\mapsto p(t,x,y)$ is in $\FF$ because $\mathcal{D(L)}$ is a dense subset of $\FF=\mathcal{D(\sqrt{-L})}$.
\section{Proof of the On-diagonal Heat Kernel Lower Bounds}\label{S:3}

For the rest of the paper, a properly exceptional set $\mathcal{N}$ is always fixed to be the same as in Proposition \ref{P:upper-bound-off-diag-Nash}. To prove Theorem \ref{main-result},  we introduce the following quantity $E_D(z, t)$ for notation convenience. Let $z\in E\setminus \mathcal{N}$ be fixed and $d$ be the intrinsic distance on $E$.  Set
\begin{equation*}\label{def-E(z,t)}
E_D(z,t):=\int_E p^2(t,z,x)\exp\left(\frac{d(z,x)^2}{Dt}\right)\mu(dx).
\end{equation*}
For fixed $z\in E$, $R>0$, we let $f_R(x):=(R-d(z,x))_+$.
By Lemma \ref{property-energy-meas} and Remark \ref{R:2.5}, $f_R$ is in  $\FF\cap L^\infty$. To establish an upper bound for $E_D(z,t)$, we first claim the next three propositions.
\begin{prop}\label{g-decreasing}
Fix  $z\in E\setminus \mathcal{N}$  and $0<T<\infty$. For any $R>0$, $D\ge 2$, the map
\begin{equation*}\label{def-f(t)}
t\mapsto \int_E p^2(t,z,x)e^{f_R(x)^2/D(t-T)}\mu(dx)
\end{equation*}
is non-increasing on  $t\in (0, T)$.
\end{prop}
\begin{proof}
In this proof, when there is no confusion, we suppress the subscript $R$ from $f_R$ for notation simplicity. Indeed we  show derivative of the map exists and is always non-positive. For this purpose,  for every $t\in (0, T)$, we write
\begin{align}
&\quad\, \frac{1}{s-t}\left(\int_E p^2(s,z,x)e^{f(x)^2/D(s-T)}\mu(dx)-\int_E p^2(t,z,x)e^{f(x)^2/D(t-T)}\mu(dx) \right)\nonumber
\\
&=\int_E \frac{1}{s-t}\left(p^2(s,z,x)e^{f(x)^2/D(s-T)}-p^2(t,z,x)e^{f(x)^2/D(s-T)}\right)\mu(dx) \nonumber
\\
&+\int_E \frac{1}{s-t}\left(p^2(t,z,x)e^{f(x)^2/D(s-T)}-p^2(t,z,x)e^{f(x)^2/D(t-T)}\right)\mu(dx) . \label{212}
\end{align}
For the first term on the right hand side of \eqref{212}, as $s\rightarrow t$, one has
\begin{align}
&\quad \lim_{s\rightarrow t}\int_E \frac{1}{s-t}\left(p(s,z,x)-p(t,z,x)\right)\left(p(s,z,x)+p(t,z,x)\right)e^{f(x)^2/D(s-T)}\mu(dx) \nonumber
\\
&=\int_E \mathcal{L} p(t,z,x)\cdot 2 p(t,z,x)e^{f(x)^2/D(t-T)}\mu(dx),\label{309}
\end{align}
because $\frac{1}{s-t}\left(p(s,x,z)-p(t,x,z)\right)\rightarrow \mathcal{L}p(t,x,z)$ strongly in $L^2$ in view of Lemma \ref{L:2.7}, $p(s,x,z)+p(t,x,z)\rightarrow 2p(t,x,z)$ also strongly in $L^2$, and $e^{f(x)^2/D(s-T)}\rightarrow e^{f(x)^2/D(t-T)}$ strongly in $L^\infty$.  To take the limit as $s\rightarrow t$ for the second term on the right hand side of \eqref{212}, with the heat kernel upper bound shown in Proposition \ref{P:upper-bound-off-diag-Nash}, it follows immediately from dominate convergence theorem that
\begin{align}
&\quad \lim_{s\rightarrow t}\int_E \frac{1}{s-t}\left(p^2(t,z,x)e^{f(x)^2/D(s-T)}-p^2(t,z,x)e^{f(x)^2/D(t-T)}\right)\mu(dx) \nonumber
\\
&=\int_E p^2(t,z,x)\frac{\partial}{\partial t}e^{f(x)^2/D(t-T)}\mu(dx).\label{310}
\end{align}
Now letting $s\rightarrow t$ in \eqref{212} by replacing the two terms in the last display of \eqref{212} with \eqref{309} and \eqref{310} respectively yields
\begin{align}
&\quad \frac{d}{dt}\int_E p^2(t,z,x)e^{f(x)^2/D(t-T)}\mu(dx) \nonumber
\\
&=\int_E 2p(t,z,x) \left(\mathcal{L}p(t,z,x)\right)e^{f(x)^2/D(t-T)}\mu(dx) +\int_E p^2(t,z,x)\frac{d}{dt}e^{f(x)^2/D(t-T)}\mu(dx)\nonumber
\\
&=-2\EE\left(p(t,z,x),\,p(t,z,x)e^{f(x)^2/D(t-T)} \right)+\int_E p^2(t,z,x)\frac{d}{dt}e^{f(x)^2/D(t-T)}\mu(dx).\label{315}
\end{align}
Note in the last ``$=$" above,  $p(t,z,x)e^{f(x)^2/D(t-T)}$ is in $\FF$ because both $p(t,z,x)$ and $e^{f(x)^2/D(t-T)}$ are in $\FF\cap L^\infty$ (see, for example, \cite[Theorem 1.4.2]{FOT}). To proceed with the computation,  we first rewrite the first term in the last display of \eqref{315} in terms of energy measure as follows:
\begin{align}
&\quad -2 \EE\left(p(t,z,x), p(t,z,x)e^{f(x)^2/D(t-T)}\right) =-2 \int_E d\Gamma\left(p(t,z,x), \,p(t,z,x)e^{f(x)^2/D(t-T)}\right) \nonumber 
\\
&=-2 \int_E p(t,z,x) d\Gamma \left(p(t,z,x), e^{f(x)^2/D(t-T)}\right)-2 \int_E e^{d(x)^2/D(t-T)}d\Gamma(p(t,z,x), p(t,z,x)) \nonumber
\\
&=-2 \int_E p(t,z,x)e^{f(x)^2/D(t-T)}d\Gamma \left(p(t,z,x),\, \frac{f(x)^2}{D(t-T)}\right)  \nonumber 
\\
&\quad -2 \int_E e^{f(x)^2/D(t-T)}d\Gamma \left(p(t,z,x), p(t,z,x)\right) \nonumber
\\
&\le 2\left(\int_E e^{f(x)^2/D(t-T)}d\Gamma \left(p(t,z,x),\,p(t,z,x) \right)\right)^{1/2} \nonumber
\\
&\quad \times \left(\int_E p^2(t,z,x)e^{f(x)^2/D(t-T)}d\Gamma \left(\frac{f(x)^2}{D(t-T)}, \frac{f(x)^2}{D(t-T)} \right)\right)^{1/2} \nonumber 
\\
&\quad -2 \int_E e^{f(x)^2/D(t-T)}d\Gamma \left(p(t,z,x), p(t,z,x)\right).\label{eq:3.4}
\end{align}
where in the third ``$=$" above we use Theorem \ref{chain-product-rule} and the last ``$\le$" is justified by \eqref{Cauchy-Schwarz-energy-measure}.  For the second term in the last display of \eqref{315}, it can first be observed that by Theorem \ref{chain-product-rule} and Lemma \ref{property-energy-meas},
\begin{align}
\frac{d}{d\mu}\Gamma \left(\frac{f(x)^2}{D(t-T)}, \frac{f(x)^2}{D(t-T)}\right)&=\frac{4f(x)^2}{D^2(t-T)^2}\frac{d}{d\mu}\Gamma (f(x), f(x)) \nonumber
\\
&\le \frac{4f(x)^2}{D^2(t-T)^2}=-\frac{4}{D}\frac{d}{dt}\left(\frac{f(x)^2}{D(t-T)}\right). \label{243}
\end{align} 
Consequently,
\begin{align}
&\quad \int_E p^2(t,z,x)\frac{d}{dt}e^{f(x)^2/D(t-T)}\mu(dx) =\int_E p^2(t,z,x)e^{f(x)^2/D(t-T)}\frac{d}{dt}\left(\frac{f(x)^2}{D(t-T)}\right)\mu(dx)\nonumber
\\
&\stackrel{\eqref{243}}{\le}-\frac{D}{4 }\int_E p^2(t,z,x)e^{f(x)^2/D(t-T)}\, d\Gamma\left(\frac{f(x)^2}{D(t-T)},\, \frac{f(x)^2}{D(t-T)}\right). \label{eq:3.6}
\end{align}
Now replacing the two terms in the last display of \eqref{315} with \eqref{eq:3.4} and \eqref{eq:3.6} gives 
\begin{align*}
&\quad \,\frac{d}{dt}\int_E p^2(t,z,x)e^{f(x)^2/D(t-T)}\mu(dx) 
\\
&\le  2\left(\int_E e^{f(x)^2/D(t-T)}d\Gamma \left(p(t,z,x),\,p(t,z,x) \right)\right)^{1/2}
\\
&\quad \times \left(\int_E p^2(t,z,x)e^{f(x)^2/D(t-T)}d\Gamma \left(\frac{f(x)^2}{D(t-T)}, \frac{f(x)^2}{D(t-T)} \right)\right)^{1/2}
\\
&\quad -2\int_E e^{f(x)^2/D(t-T)}d\Gamma \left(p(t,z,x),\,p(t,z,x)\right)
\\
&\quad -\frac{D}{4 }\int_E p^2(t,z,x)e^{f(x)^2/D(t-T)}\, d\Gamma\left(\frac{f(x)^2}{D(t-T)},\, \frac{f(x)^2}{D(t-T)}\right)\le 0,
\end{align*}
which is justified by the second  inequality of \eqref{Cauchy-Schwarz-energy-measure} because $D\ge 2$. 
\end{proof}

The next proposition says the on-diagonal heat kernel is monotonically non-increasing in $t$.

\begin{prop}\label{on-diag-HK-mono-dec}
$p(t,z, z)$ is non-increasing in $t\in (0, \infty)$, for all $z\in E\setminus \mathcal{N}$.
\end{prop}
\begin{proof} This follows from
\begin{align*}
\frac{d}{dt}p(t,z, z)&=\frac{d}{dt}\int_E  p(t/2, z,x)^2 \mu(dx)
\\
&=2\int_E \mathcal{L} p(t/2,z,x)p(t/2,z,x)\mu (dx)
\\
&=-2\,\EE\left(p(t/2, z,x), p(t/2,z,x)\right) \le 0,
\end{align*}
where the second ``$=$" can be justified in an analogous manner to the proof of Proposition \ref{g-decreasing}, in view of  the fact that the strong $L^2$-derivative $\frac{\partial}{\partial t}p(t, z, x)$ exists on $(0, \infty)$ and equals $\mathcal{L} p(t,z,x)$.
\end{proof}

The following proposition is comparable  to  the integral estimate in \cite[Lemma 3.1]{gowri}.
\begin{prop}\label{upb-integral-psquare}
Fix $z\in E\setminus \mathcal{N}$. Assume that for all $r>0$, $\mu(B(z,r))\le v(r)$, where  $v(r)$ is a continuous monotonically increasing function satisfying doubling property.
Suppose  for some $C_1>0$, $T\in (0, \infty]$,
\begin{equation*}
p(t,z,z)\le \frac{C_1}{v(\sqrt{t})}, \quad t\in (0, T).
\end{equation*}
There exist constants $C_2, C_3>0$ such that 
\begin{equation}\label{ineq-upb-integral-psquare}
\int_{E\setminus B(z, R)} p^2(t, z, x)\mu(dx) \le \frac{C_2}{v(\sqrt{t})} e^{-C_3R^2/t}, \quad \text{for all } t\in (0, T), R>0.
\end{equation}

\end{prop}

\begin{proof}
In Proposition \ref{g-decreasing}, taking $D=2$ yields that  for any $0<\tau<t<\tau'<T$ and any $R>0$,
\begin{equation*}
\int_E p^2(t,z,x)e^{f_R(x)^2/2(t-\tau')}\mu(dx) \le \int_E p^2(\tau, z,x)e^{f_R(x)^2/2(\tau-\tau')}\mu(dx).
\end{equation*} 
Rewriting each side above as a sum of two integrals over $B(z, R)$ and $E\setminus B(z, R)$ respectively yields that for $\rho<R$,
\begin{align*}
&\quad\, \int_{E\setminus B(z,R)} p^2(t,z,x)\mu(dx) 
\\
&\le \int_{B(z,R)} p^2(\tau, z,x)e^{(R-d(z,x))^2/2(\tau-\tau')}\mu(dx)+\int_{E\setminus B(z, R)} p^2(\tau, z, x)\mu(dx)
\\
&\le \int_{B(z,\rho)} p^2(\tau, z,x)e^{(R-d(z,x))^2/2(\tau-\tau')}\mu(dx)+\int_{E\setminus B(z, \rho)} p^2(\tau, z, x)\mu(dx).
\end{align*}
We observe that since $\tau<\tau'$, the exponential term involved in the last display is bounded by 
$$
 \exp \left(-\frac{(R-\rho)^2}{2(\tau'-\tau)}\right) \int_{B(z, \rho)} p^2(\tau, x,z)\mu(dx).
$$
Therefore, by letting $\tau'\rightarrow t+$ and using semigroup property, we get
\begin{align*}
&\quad \,\int_{E\setminus B(z,R)} p^2(t,z,x)\mu(dx)
\\
& \le \exp \left(-\frac{(R-\rho)^2}{2(t-\tau)}\right) \int_{B(z, \rho)} p^2(\tau, z,x)\mu(dx) +\int_{E\setminus B(z, \rho)} p^2(\tau, z,x)\mu(dx)
\\
&\le \exp \left(-\frac{(R-\rho)^2}{2(t-\tau)}\right) \int_E p^2(\tau, z,x)\mu(dx) +\int_{E\setminus B(z, \rho)} p^2(\tau, z,x)\mu(dx)
\\
&\le \exp\left(-\frac{(R-\rho)^2}{2(t-\tau)}\right)p(2\tau, z, z)+\int_{E\setminus B(z, \rho)}p^2(\tau, z, x)\mu(dx)
\\
&\le \exp\left(-\frac{(R-\rho)^2}{2(t-\tau)}\right)p(\tau, z, z)+\int_{E\setminus B(z, \rho)}p^2(\tau, z, x)\mu(dx)
\\
&\le \frac{1}{v(\sqrt{\tau})} \exp\left(-\frac{(R-\rho)^2}{2(t-\tau)}\right)+\int_{E\setminus B(z, \rho)}p^2(\tau, z, x)\mu(dx),
\end{align*}
where the second last ``$\le $" is justified by Proposition \ref{on-diag-HK-mono-dec}. Now we consider two decreasing sequences $t_k=t\cdot 2^{-k}$ and $R_k=\left(\frac{1}{2}+\frac{1}{k+2}\right)R$ for  $k=0, 1, \cdots$. Replacing $t, \tau, R, \rho$ with $t_{k-1}, t_k, R_{k-1}, R_k$ gives that for $k\ge 1$,
\begin{align*}
&\quad \int_{E\setminus B(z,R_{k-1})} p^2(t_{k-1},z,x)\mu(dx)
\\
& \le \frac{1}{v(\sqrt{t_k})} \exp\left(-\frac{(R_{k-1}-R_k)^2}{2(t_{k-1}-t_k)}\right)+\int_{E\setminus B(z, R_k)}p^2(t_k, z, x)\mu(dx)
\end{align*}
Summing the above inequality in $k$ from $1$ to $n$ and canceling the common terms from both sides gives
\begin{equation}\label{1223}
\int_{E\setminus B(z,R)} p^2(t,z,x)\mu(dx) \le \sum_{k=1}^n \frac{1}{v(\sqrt{t_k})} \exp\left(-\frac{(R_{k-1}-R_k)^2}{2(t_{k-1}-t_k)}\right)+\int_{E\setminus B(z, R_n)}p^2(t_n, z, x)\mu(dx).
\end{equation}
Observing that $t_n\downarrow 0$ and $R_n\downarrow R/2$, in view of Proposition \ref{P:upper-bound-off-diag-Nash}, we get 
\begin{align}
\lim_{n\rightarrow \infty}\int_{E\setminus B(z, R_n)}p^2(t_n, z, x)\mu(dx) &\le \lim_{n\rightarrow \infty}\int_{E\setminus B(z, R/2)}p^2(t_n, z, x)\mu(dx) \nonumber
\\
& =\int_{E\setminus B(z, R/2)}\lim_{n\rightarrow \infty}p^2(t_n, z, x)\mu(dx)=0. \label{75424}
\end{align}
Hence, letting $n\rightarrow \infty$ in \eqref{1223} shows
\begin{equation*}
\int_{E\setminus B(z,R)} p^2(t,z,x)dx \le \sum_{k=1}^\infty  \frac{1}{v(\sqrt{t_k})} \exp\left(-\frac{(R_{k-1}-R_k)^2}{2(t_{k-1}-t_k)}\right).
\end{equation*}
By the doubling property of the function $v$, it holds for some $c_1>0$ that 
\begin{equation*}
v(\sqrt{t})\le A^{k/2+1} v(\sqrt{t_k}) \le   e^{c_1 k} v(\sqrt{t_k}),
\end{equation*}
where $A$ is the same as in Theorem \ref{main-result}. It follows that 
\begin{align}
\int_{E\setminus B(z, R)}p^2(t, z, x)dx &\le \sum_{k=1}^\infty  \frac{1}{v(\sqrt{t_k})} \exp\left(-\frac{(R_{k-1}-R_k)^2}{2(t_{k-1}-t_k)}\right) \nonumber
\\
&\le  \sum_{k=1}^\infty \frac{e^{c_1k}}{v(\sqrt{t})}\exp \left(-\frac{\left(\frac{1}{k+1}-\frac{1}{k+2}\right)^2R^2}{2t\cdot 2^{-k}}\right)\nonumber
\\
&\le \sum_{k=1}^\infty \frac{1}{v(\sqrt{t})}\exp \left(c_1k-\frac{2^{k-1}R^2}{(k+2)^4t}\right). \label{119}
\end{align}
We select constants $c_2, c_3>0$ such that 
\begin{equation*}
\frac{2^{k-1}}{(k+2)^4} >c_2k+c_3, \quad \text{for all } k\ge 1. 
\end{equation*}
Therefore, when $R^2/t>2c_1/c_2$, the quantity inside the brackets of  the last display of \eqref{119}  can be bounded by
\begin{align*}
c_1k-\frac{2^{k-1}R^2}{(k+2)^4t} &<c_1k-(c_2k+c_3) \frac{R^2}{t }
\\
&< c_1k-c_2k\cdot\frac{2c_1}{c_2}-c_3\frac{R^2}{t}<-c_1k-c_3\frac{R^2}{t}.
\end{align*}
i.e., when $R^2/t>2c_1/c_2$, there exists some $c_4>0$ such that
\begin{equation*}
\int_{E\setminus B(z, R)}p^2(t, z, x)dx \le\sum_{k=1}^\infty \frac{1}{v(\sqrt{t})}\exp\left(-c_1k-c_3\frac{R^2}{t}\right)\le \frac{c_4}{v(\sqrt{t})}e^{-c_3R^2/t}.
\end{equation*}
 On the on the other hand, when $R^2/t\le 2c_1/c_2$, due to its boundedness we immediately conclude that there exist $c_5, c_6>0$ such that
\begin{equation*}
\int_{E\setminus B(z, R)}p^2(t, z, x)dx \le \int_E p^2(t, z, x)dx \le p(2t, z, z)\le p(t,z,z) \le \frac{c_5}{v(\sqrt{t})}\le \frac{c_6}{v(\sqrt{t})}e^{-R^2/t},
\end{equation*}
where the third ``$\le$" from left is due to the monotonicitiy of $p(t,z,z)$ stated in Proposition \ref{on-diag-HK-mono-dec}. The proof is thus complete by combining both cases above.
\end{proof}

We finally establish the following upper bound for $E_D(z, t)$ before proving the main theorem. 
 
\begin{lem}
Fix $z\in E\setminus \mathcal{N}$. Assume that for all $r>0$, $\mu(B(z,r))\le v(r)$, where  $v(r)$ is a continuous monotonically increasing function satisfying doubling property.
Suppose  for some $C_4>0$, $T\in (0, \infty]$,
\begin{equation}\label{UPE-L:2.4}
p(t,z,z)\le \frac{C_4}{v(\sqrt{t})}, \quad t\in (0, T).
\end{equation}
 Then there exist some $C_5>0$ and $D_0>0$ such  that for any $D>D_0$, 
\begin{equation*}\label{upb-E_D}
E_D(z, t)\le \frac{C_5}{v(\sqrt{t})}, \quad t\in (0, T).
\end{equation*}
 \end{lem}
\begin{proof}
Note that $E_D$ is decreasing in  $D$, therefore it suffices to show the conclusion for some $D>0$. Fix any $D>5/C_3$, where $C_3 $ is the same as in Proposition \ref{upb-integral-psquare}, by choosing $R=\sqrt{Dt}$ we decompose $E_D(z, t)$ as follows:
\begin{align}
\int_E p^2(t, z, x)\exp\left(\frac{d(z, x)^2}{Dt}\right)\mu(dx)
&=\int_{B(z, R)}p^2(t, z, x)\exp\left(\frac{d(z, x)^2}{Dt}\right) \mu(dx) \nonumber
\\
&+\sum_{k=0}^\infty \int_{2^kR\le d(z,x)\le 2^{k+1}R}p^2(t, z, x)\exp\left(\frac{d(z, x)^2}{Dt}\right)\mu(dx). \label{552}
\end{align}
For the first term on the right hand side of \eqref{552}, since $R=\sqrt{Dt}$, it holds from the semigroup property and \eqref{UPE-L:2.4} that for some $c_1>0$,
\begin{align*}
\int_{B(z, R)}p^2(t, z, x)\exp\left(\frac{d(z, x)^2}{Dt}\right)\mu(dx) &\le e^{R^2/Dt}\int_E p^2(t, z, x)\mu(dx) 
\\
&\le e\cdot  p(2t, z, z ) 
\le e\cdot p(t,z,z)\le
\frac{c_1}{v(\sqrt{t})},
\end{align*}
where the second last inequality is again justified by the monotonicity of $p(t,z,z)$ on $t\in (0, \infty)$.
For the summation term in \eqref{552}, observing that $D>5/C_3$ for $C_3$ the same as  in  Proposition \ref{upb-integral-psquare}, we get that there exists some $c_2>0$ such that
\begin{align*}
&\quad \sum_{k=0}^\infty \int_{2^kR\le d(z,x)\le 2^{k+1}R}p^2(t, z, x)\exp\left(\frac{d(z, x)^2}{Dt}\right) \mu(dx)
\\
&\le \sum_{k=0}^\infty \exp \left(\frac{4^{k+1}R^2}{Dt}\right)\int_{d(z,x)\ge 2^{k}R} p^2(t,z,x)\mu(dx) 
\\
&\le \sum_{k=0}^\infty \exp \left(\frac{4^{k+1}R^2}{Dt}\right) \frac{c_2}{v(\sqrt{t})}\exp \left(-\frac{5\cdot 4^kR^2}{Dt}\right)
\\
&\le  \frac{c_2}{v(\sqrt{t})}\sum_{k=0}^\infty \exp \left(-\frac{4^{k}R^2}{Dt}\right)
\\
&=  \frac{c_2}{v(\sqrt{t})}\sum_{k=0}^\infty e^{-4^k}=\frac{c_3}{v(\sqrt{t})}, \quad t\in (0, T).
\end{align*}
Combining the computation for both terms on the right hand side of  \eqref{552}   yields that there exists some $c_4>0$ such that
\begin{equation*}
E_{D}(z, t)\le  \frac{c_4}{v(\sqrt{t})}, \quad \text{for all }t\in (0, T).
\end{equation*}
Therefore the proof is complete by choosing $D_0>5/C_3$ where $C_3$ the same as in Proposition \ref{upb-integral-psquare}.

\end{proof}

To proceed, we introduce another quantity for notation simplicity. For any $D>0$, $R>0$, set 
\begin{equation}\label{def-I(z,R)}
I_D(z, t, R):=\int_{E\setminus B(z,R)}e^{-d(x,z)^2/Dt}\mu(dx).
\end{equation}
It follows from H\"{o}lder's inequality that for any $z\in E\setminus \mathcal{N}$ and any $R>0$,
\begin{align}
&\quad \left(\int_{E\setminus B(z, R)} p(t/2, z, x)\mu(dx)\right)^{2} \nonumber
\\
&\le \int_{ E\setminus B(z, R)} p^2(t/2, z, x)e^{d(x,z)^2/Dt}\mu(dx)\int_{E\setminus B(z, R)} e^{-d(x,z)^2/Dt}\mu(dx) \nonumber
\\
&\le E_D(z, t/2)\int_{E\setminus B(z, R)} e^{-d(x,z)^2/Dt}\mu(dx)=E_D(z, t/2)I_D(z,t, R). \label{213}
\end{align}

Now we are in the position to prove the following main theorem:
\begin{thm}\label{HKLE}
Let $(\EE, \FF)$ be a strongly local regular symmetric Dirichlet form satisfying Assumption  \ref{strong-regularity} and Nash-type inequality \eqref{Nash-inequality-I}. Fix $z\in E\setminus \mathcal{N}$ where $\mathcal{N}$ is a properly exceptional set. Assume that for all $r>0$, $\mu(B(z,r))\le v(r)$, where  $v(r)$ is a continuous monotonically increasing function satisfying doubling property in the following sense: There exists some $A>0$ such that 
\begin{equation*}\label{doubling-P:2.3}
v(2r)\le Av(r), \quad \text{for all }r>0.
\end{equation*}
Suppose  for some $C_6>0$, $T\in (0, \infty]$,
\begin{equation*}
p(t,z,z)\le \frac{C_6}{v(\sqrt{t})}, \quad t\in (0, T).
\end{equation*}
Then there exists $C_7>0$ such that for all $t\in (0,T)$,
\begin{equation*}
p(t, z,z)\ge \frac{C_7}{v(\sqrt{t})}.
\end{equation*}
\end{thm}
\begin{proof}
Let $\Omega:= B(z, R)$ where $R>0$ will be determined later. $\mu(\Omega)\le v(R)$ by the assumption.  In view of the symmetry and the semigroup property of $p(t,x,y)$, 
\begin{align*}
p(t, z,z)&=\int_E p^2(t/2, z,x)\mu(dx)\ge \int_\Omega p^2(t/2, z, x) \mu(dx) \ge  \frac{1}{\mu(\Omega )}\left( \int_\Omega p(t/2, z, x)\mu(dx)\right)^2
\\
&=\frac{1}{\mu(\Omega )}\left(1- \int_{E\setminus \Omega} p(t/2, z, x)\mu(dx)\right)^2 \ge \frac{1}{v(R)}\left(1- \int_{E\setminus \Omega} p(t/2, z, x)\mu(dx)\right)^2.
\end{align*}
To give an upper bound for \eqref{213}:
\begin{equation}\label{Holder}
\left(\int_{E\setminus B(z, R)} p(t/2, z, x)\mu(dx)\right)^2\le E_D(z,t)I_D(z, t,R), \quad \text{for all }t\in (0, T),
\end{equation}
we first select and fix a constant $D>\max\{D_0, 2\}$ where $D_0$ is  the same as in Lemma \ref{upb-E_D}. By the doubling property of $v(\cdot)$, there exists some constant $B>1$ such that $v(Dr)\le Bv(r)$, for all $r>0$. We thus let $R=a\sqrt{t}$ and $R_k=D^kR$, $k=0, 1, 2, \cdots$, where the constant $a>0$ is chosen to satisfy
\begin{equation}\label{condition-constant-a}
\frac{a^2}{D}\ge 2 \ln B.
\end{equation}
It follows that $v(R_{k+1})\le B^k v(R)$. Observing that $D^{2k}\ge k+1$ for all $k\ge 0$, we  have
\begin{align*}
I_D(z, t, R)&=\sum_{k=0}^\infty \int_{B(z, R_{k+1})\setminus B(z, R_k)}e^{-d(x, z)^2/Dt}\mu(dx)
\\
&\le \sum_{k=0}^\infty \exp \left(-\frac{R_k^2}{Dt}\right)\mu(B(z, R_{k+1}))
\\
&\le \sum_{k=0}^\infty \exp \left(-
\frac{R^2_{k}}{Dt}\right)B^{k}v(R)
\\
&=v(R)\sum_{k=0}^\infty \exp\left(-D^{2k}\;\frac{R^2}{Dt}+k\ln B\right)
&
\\
&\stackrel{\eqref{condition-constant-a}}{\le} v(R)\sum_{k=0}^\infty \exp\left(-D^{2k}\;\frac{a^2}{D}+k\frac{a^2}{2D}\right)
\\
&\le v(R)\sum_{k=0}^\infty\exp \left(-\frac{(k+1)}{2D}a^2\right)\le \frac{v(R)}{e^{a^2/2D}-1}=\frac{v(a\sqrt{t})}{e^{a^2/2D}-1}. 
\end{align*}

Combining this with \eqref{Holder},  we conclude from Lemma \ref{upb-E_D} that for some fixed large constant $D$, there exists some $C=C(D)>0$, such that for any $a$ satisfying \eqref{condition-constant-a}, 
\begin{equation}
\left(\int_{E\setminus B(z, R)} p(t/2, z, x)\mu(dx)\right)^2 \le \frac{C}{v(\sqrt{t})}\cdot \frac{v(a\sqrt{t})}{e^{a^2/2D}-1} \le \frac{C\cdot A^{[\log_2 a]+1}}{e^{a^2/2D}-1}, \quad \text{for all }t\in (0, T).
\end{equation}
The rightmost term above can be made less than $1/4$ (indeed, arbitrarily small) by selecting $a$ sufficiently large in \eqref{condition-constant-a}. Therefore, for such a constant $a$ and $R=a\sqrt{t}$,  it holds for some $c_1>0$ that
\begin{align*}
p(t, z,z)&\ge \frac{1}{v(R)}\left(1- \int_{E\setminus \Omega} p(t/2, z, x)\mu(dx)\right)^2
\\
&\ge \frac{1}{v(R)}\left(1-\sqrt{1/4}\right)^2
\\
&\ge \frac{1}{2v(a\sqrt{t})} \ge \frac{c_1}{v(\sqrt{t})} \qquad \text{on }t\in (0, T),
\end{align*}
where the last ``$\ge$" is again due to the doubling property of $v(\cdot)$, since $a$ has been fixed. This completes the proof. 

\end{proof}


\vskip 0.3truein

\noindent {\bf Shuwen Lou}

\smallskip \noindent
Department of Mathematics, Statistics, and Computer Science,

\noindent
University of Illinois at Chicago,
Chicago, IL 60607, USA

\noindent
E-mail:  \texttt{slou@uic.edu}

\end{document}